\documentclass[11pt,reqno]{amsart}
\usepackage{tikz}
\usetikzlibrary{decorations.pathmorphing}
\usetikzlibrary{decorations.fractals}
\newcommand{\df}{\dfrac}

\newcommand{\x}{\xi}

\renewcommand{\(}{\left\(}
\renewcommand{\)}{\right\)}
\renewcommand{\[}{\left\[}
\renewcommand{\]}{\right\]}
\renewcommand{\i}{\infty}

\numberwithin{equation}{section}
 \theoremstyle{plain}
 
\newtheorem{theorem}{Theorem}[section]
\newtheorem{lemma}[theorem]{Lemma}

\newtheorem{definition}[theorem]{Definition}

   \makeatletter
\def\proof{\@ifnextchar[{\@oproof}{\@nproof}}
\def\@oproof[#1][#2]{\trivlist\item[\hskip\labelsep\textit{#2 Proof of\
#1.}~]\ignorespaces}
\def\@nproof{\trivlist\item[\hskip\labelsep\textit{Proof.}~]\ignorespaces}
\def\endproof{\qed\endtrivlist}
\makeatother

\begin{document}

\title[Generalized higher order spt-functions]{Generalized higher order spt-functions}
%{\Large }

\author{Atul Dixit}
\address{Department of Mathematics, University of Illinois, Urbana, IL 61801,
USA} \email{aadixit2@illinois.edu}

\author{Ae Ja Yee}
\address{Department of Mathematics, Penn State University, University Park, PA 16802,
USA} \email{yee@math.psu.edu}

\dedicatory{Dedicated to our friends, Mourad Ismail and Dennis Stanton}

\footnotetext[1]{The second author was partially supported by National Security Agency Grant H98230-10-1-0205 and by the Australian Research Council.}


\footnotetext[3]{2000 AMS Classification Numbers: Primary, 11P81; Secondary,05A17}

%\vspace*{0.5in}
%\begin{center}

%{\bf \Large An analogue of Spt }\\[5mm]

%{\footnotesize ATUL DIXIT and AE JA YEE\footnotemark[1]}\\[3mm]
%\end{center}
\maketitle

\begin{abstract}
We give a new generalization of the spt-function of G.E.~Andrews, namely $\textup{Spt}_j(n)$, and give its combinatorial interpretation in terms of successive lower-Durfee squares. We then generalize the higher order spt-function $\textup{spt}_k(n)$, due to F.G.~Garvan, to ${}_j\textup{spt}_{k}(n)$, thus providing a two-fold generalization of $\textup{spt}(n)$, and give its combinatorial interpretation.
\end{abstract}

\section{Introduction}
Two fundamental statistics in the theory of partitions are Dyson's rank \cite{dys} and the Andrews-Garvan crank \cite{ang}. While the rank of a partition is defined as the largest part minus the number of parts, the crank is defined as the largest part if the partition contains no ones, and otherwise as the number of parts larger than the number of ones minus the number of ones. Dyson observed \cite{dys} that the rank of a partition could explain two of Ramanujan's famous partition congruences, namely,
\begin{align}
p(5n+4)&\equiv 0\hspace{1mm}(\textup{mod}\hspace{1mm} 5)\label{pc5}\\
p(7n+5)&\equiv 0\hspace{1mm}(\textup{mod}\hspace{1mm}7)\label{pc7},
\end{align}
but not the third one, i.e., 
\begin{equation}
p(11n+6)\equiv 0\hspace{1mm}(\textup{mod}\hspace{1mm}11)\label{pc11}.
\end{equation}
This led him to hypothesize the existence of another statistic, namely the crank, though its discovery \cite{ang} was not made until 1988.

Let $N(m,n)$ denote the number of partitions of $n$ with rank $m$. Then the rank generating function $R(z,q)$ is given by
\begin{equation}\label{rankgf}
R(z,q)=\sum_{n=0}^{\infty}\sum_{m=-\infty}^{\infty}N(m,n)z^mq^n=\sum_{n=1}^{\infty}\frac{q^{n^2}}{(zq)_{n}(z^{-1}q)_{n}}.
\end{equation}
Here, and in the sequel, we employ the standard notation
\begin{align*}
(A)_0 &:=(A;q)_0 =1, \qquad \\
(A)_n &:=(A;q)_n  = (1-A)(1-Aq)\cdots(1-Aq^{n-1}),
\qquad n \geq 1, \\
(A)_{\infty} &:=(A;q)_{\i}  = \lim_{n\to\i}(A;q)_n, \qquad |q|<1.
\end{align*}
Similarly, if $M(m,n)$ denote the number of partitions of $n$ with crank $m$, then the crank generating function $C(z,q)$ is given by
\begin{equation}\label{crankgf}
C(z,q)=\sum_{n=0}^{\infty}\sum_{m=-\infty}^{\infty}M(m,n)z^mq^n=\frac{(q)_{\infty}}{(zq)_{\infty}(z^{-1}q)_{\infty}}.
\end{equation}
The generating functions of $N(m,n)$ and $M(m,n)$ are respectively given by
\begin{align}
\sum_{n=0}^{\infty} N(m,n) q^n & =\frac{1}{(q)_{\infty}} \sum_{n=1}^{\infty} (-1)^{n-1} q^{n(3n-1)/2+ |m|n}(1-q^n) \label{rank}
\end{align}
and 
\begin{align}
\sum_{n=0}^{\infty} M(m,n) q^n & =\frac{1}{(q)_{\infty}} \sum_{n=1}^{\infty} (-1)^{n-1} q^{n(n-1)/2+ |m|n}(1-q^n). \label{crank}
\end{align}
In \cite{atkin-garvan}, Atkin and Garvan introduced the rank and crank moments which are defined by
\begin{equation}\label{rankm}
N_{t}(n):=\sum_{m=-\infty}^{\infty}m^tN(m,n)
\end{equation}
and
\begin{equation}\label{crankm}
M_{t}(n):=\sum_{m=-\infty}^{\infty}m^tM(m,n)
\end{equation}
respectively. The above series are really finite series with $m$ ranging from $-n$ to $n$. The odd moments of rank and crank equal zero. This follows from the facts that $N(m,n)=N(-m,n)$ and $M(m,n)=M(-m,n)$, which in turn are easy consequences of (\ref{rank}) and (\ref{crank}). 

Recently, Andrews \cite{andrews} defined the smallest part function $\textup{spt}(n)$ as the total number of appearances of the smallest parts in all the partitions of $n$ and showed that
\begin{align}\label{sptpn}
\textup{spt}(n)=np(n)-\frac{1}{2}N_2(n),
\end{align}
where $p(n)$ is the number of partitions of $n$ and  $N_2(n)$ is the second Atkin-Garvan rank moment defined in (\ref{rankm}). Andrews proved (\ref{sptpn}) by obtaining an identity involving the generating functions of $\textup{spt}(n)$, $np(n)$ and $N_{2}(n)$, i.e.,
\begin{align}\label{andrerel}
\sum_{m=1}^{\infty}\frac{q^{m}}{(1-q^{m})^{2}(q^{m+1};q)_{\infty}}=\frac{1}{(q;q)_{\infty}}\sum_{n=1}^{\infty}\frac{nq^n}{1-q^n}+\frac{1}{(q;q)_{\infty}}\sum_{n=1}^{\infty}\frac{(-1)^nq^{n(3n+1)/2}(1+q^n)}{(1-q^n)^2}.
\end{align}
To see his derivation, we first need Watson's $q$-analogue of Whipple's theorem \cite[Equation (2.2)]{andrews} given by
\begin{multline}\label{watson87}
_8\phi_7\left[\begin{matrix} a,& q\sqrt{a},& -q\sqrt{a},& b, &
c, &d, & e, & q^{-N}\\
 &\sqrt{a}, & -\sqrt{a},& \df{a
q}{b}, & \df{a q}{c}, & \df{a q}{d}, & \df{a q}{e}, & a
q^{N+1}\end{matrix}\,; q,
 \df{a^2q^{N+2}}{bcde}\right] \\
=\df{(a q)_N\left(\df{a q}{de}\right)_N}
{\left(\df{a q}{d}\right)_N\left(\df{a q}{e}\right)_N}
\ _4\phi_3\left[\begin{matrix}\df{a q}{bc},d, e,q^{-N}\\ \df{a q}{b},
\df{a q}{c}, \df{de q^{-N}}{a}\end{matrix}\,; q, q\right],
\end{multline}
where
\begin{align*}
_r\phi_{r-1}\left[\begin{matrix} a_1, a_2, \ldots, a_r\\
  b_1,  b_2, \ldots, b_{r-1} \end{matrix}\,; q,
z \right] =\sum_{n=0}^{\infty} \frac{(a_1)_n (a_2)_n \cdots (a_r)_n}{(q)_n (b_1)_n \cdots (b_{r-1})_n} z^n.
\end{align*}
Andrews obtained (\ref{andrerel}) by first specializing $d=e^{-1}=z$, then letting $b, c, N\to\infty$ and $a\to 1$ in (\ref{watson87}), thereby obtaining 
\begin{equation}\label{watspl}
\sum_{n=0}^{\infty}\frac{(z)_{n}(z^{-1})_{n}q^n}{(q)_{n}}=\frac{(zq)_{\infty}(z^{-1}q)_{\infty}}{(q)_{\infty}^{2}}\left(1+\sum_{n=1}^{\infty}\frac{(-1)^nq^{n(3n+1)/2}(1+q^n)(z){n}(z^{-1})_{n}}{(zq)_{n}(z^{-1}q)_{n}}\right),
\end{equation}
then taking the second derivative with respect to $z$ of both sides of (\ref{watspl}), and then letting $z=1$.

Now Andrews \cite{propro} has obtained a generalization of (\ref{watson87}) for $j\geq 1$ which is as follows:
{\allowdisplaybreaks\begin{align}\label{andrewswatsong}
 &_{2j+6} \phi_{2j+5} \left[\begin{matrix}
  a, & q\sqrt{a}, & -q\sqrt{a}, & b_{1}, & c_{1}, %& b_{2}, & c_{2}, 
  & \cdots, & b_{j+1}, & c_{j+1}, q^{-N} \\
  \sqrt{a}, & -\sqrt{a}, & \frac{aq}{b_{1}}, & \frac{aq}{c_{1}}, %& \frac{aq}{b_{2}}, & \frac{aq}{c_{2}}, 
  & \cdots, & \frac{aq}{b_{j+1}}, & \frac{aq}{c_{j+1}},& aq^{N+1}\end{matrix}\, ;q, \frac{a^{j+1}q^{N+j+1}}{b_{1}\cdots b_{j+1}c_{1}\cdots c_{j+1}} \right]\nonumber\\
  &=\frac{(aq)_{N}(\frac{aq}{b_{j+1}c_{j+1}})_{N}}{(\frac{aq}{b_{j+1}})_{N}(\frac{aq}{c_{j+1}})_{N}}\sum_{m_{1},\cdots, m_{j}\geq 0}\frac{(\frac{aq}{b_{1}c_{1}})_{m_{1}}(\frac{aq}{b_{2}c_{2}})_{m_{2}}\cdots (\frac{aq}{b_{j}c_{j}})_{m_{j}}(b_{2})_{m_{1}}(c_{2})_{m_{1}}(b_{3})_{m_{1}+m_{2}}(c_{3})_{m_{1}+m_{2}}}{(q)_{m_{1}}(q)_{m_{2}}\cdots (q)_{m_{j}}(\frac{aq}{b_{1}})_{m_{1}}(\frac{aq}{c_{1}})_{m_{1}}(\frac{aq}{b_{2}})_{m_{1}+m_{2}}(\frac{aq}{c_{2}})_{m_{1}+m_{2}}}\nonumber\\
  &\times\frac{\cdots (b_{j+1})_{m_{1}+\cdots+m_{j}}(c_{j+1})_{m_{1}+\cdots+m_{j}}(q^{-N})_{m_{1}+\cdots+m_{j}}(aq)^{m_{j-1}+2m_{j-2}+\cdots+(j-1)m_{1}}q^{m_{1}+\cdots+m_{j}}}{\cdots (\frac{aq}{b_{j}})_{m_{1}+\cdots+m_{j}}(\frac{aq}{c_{j}})_{m_{1}+\cdots+m_{j}}(\frac{b_{j+1}c_{j+1}}{aq^{N}})_{m_{1}+\cdots+m_{j}}(b_{2}c_{2})^{m_{1}}(b_{3}c_{3})^{m_{1}+m_{2}}\cdots (b_{j}c_{j})^{m_{1}+\cdots+m_{j-1}}}.
 \end{align}}
It then seems natural to generalize Andrews' approach by specializing (\ref{andrewswatsong}) to obtain an identity similar to (\ref{watspl}) and then taking second derivatives with respect to $z$ of both sides of this identity to obtain a generalization of (\ref{sptpn}). This may then lead us to a generalization of Andrews' spt-function. In this paper, we show that this is indeed the case, i.e., we obtain a generalization of $\textup{spt}(n)$ (which we denote by $\textup{Spt}_{j}(n)$), and of (\ref{sptpn}). We also provide a combinatorial interpretation of $\textup{Spt}_j(n)$.

To see how (\ref{sptpn}) can be generalized, we first need to generalize $N(m,n)$. This was done by Garvan \cite{garvan_successive} who generalized Dyson's rank to $j$-rank which is defined as follows. For a partition $\pi$, define $n_{1}(\pi), n_{2}(\pi)\cdots$ to be the sizes of the successive Durfee squares of $\pi$. Then the $j$-rank of the partition $\pi$ is defined as the `number of columns in the Ferrers graph of $\pi$ which lie to the right of the first Durfee square and whose length $\leq n_{j-1}(\pi)$ minus the number of parts of $\pi$ that lie below the $(j-1)$st-Durfee square'. When %$j=1$, this gives the Andrews-Garvan crank and when 
$j=2$, this gives Dyson's rank. Let $N_{j}(m,n)$ be the number of partitions of $n$ with at least $j-1$ successive Durfee squares whose $j$-rank is equal to $m$. Then Garvan showed that for $j\ge 2$, 
\begin{align}
\sum_{n=0}^{\infty} N_{j}(m,n)q^n =\frac{1}{(q)_{\infty}} \sum_{n=1}^{\infty} (-1)^{n-1} q^{n((2j-1)n-1)/2+|m|n}(1-q^n). \label{Nk}
\end{align}
and 
\begin{align}\label{R_k}
R_j(z,q)&=\sum_{n=1}^{\infty} \sum_{m=-\infty}^{\infty}  N_j(m,n)z^mq^n  \notag\\
&=\sum_{n_{j-1}\ge \cdots \ge n_1\ge 1 } \frac{q^{n_1^2+\cdots +n_{j-1}^2}}{(q)_{n_{j-1}-n_{j-2}} \cdots (q)_{n_2-n_1} (zq)_{n_1} (z^{-1}q)_{n_1}} \notag\\
&= \frac{z}{(q)_{\infty}} \sum_{n=-\infty \atop n\neq 0}^{\infty} (-1)^{n-1} q^{n((2j-1)n+1)/2} \frac{1-q^n}{1-zq^n}.
\end{align}
Now (\ref{Nk}) readily implies that $N_{j}(m,n)=N_{j}(-m,n)$. Define the $j$-rank moment ${}_jN_{t}(n)$, analogous to (\ref{rankm}) and (\ref{crankm}), by
\begin{equation}\label{3rm}
{}_jN_{t}(n):=\sum_{m=-\infty}^{\infty}m^tN_{j}(m,n).
\end{equation}
The above series is really a finite series with $m$ ranging from $-n$ to $n$. It is easy to see that for odd $t$, we have ${}_jN_{t}(n)=0$. When $j=2$, ${}_2N_{t}(n)$ is the same as the Atkin-Garvan rank moment $N_{t}(n)$. Also, $j=1$ corresponds to the crank moment $M_{t}(n)$, i.e., ${}_1N_{t}(n)=M_{t}(n)$.

We show in Section 2 that (\ref{sptpn}) can be generalized to
\begin{align}\label{sptpng}
\textup{Spt}_j(n)=np(n)-\frac{1}{2}{}_{j+1}N_2(n),
\end{align}
where $\textup{Spt}_j(n)$ is defined in (\ref{sptjn}) below.
%Note that $\textup{Spt}_1(n)=\textup{spt}(n)$.

%Garvan also showed that
%\begin{align}
%R_k(z,q)&=\sum_{n=1}^{\infty} \sum_{m=-n}^{n}  N_k(m,n)z^mq^n  \notag\\
%&=\sum_{n_{k-1}\ge \cdots \ge n_1\ge 1 } \frac{q^{n_1^2+\cdots +n_{k-1}^2}}{(q)_{n_{k-1}-n_{k-2}} \cdots (q)_{n_2-n_1} (zq)_{n_1} (z^{-1}q)_{n_1}} \notag\\
%&= \frac{z}{(q)_{\infty}} \sum_{n=-\infty \atop n\neq 0}^{\infty} (-1)^{n-1} q^{n((2k-1)n+1)/2} \frac{1-q^n}{1-zq^n} \label{R_k}
%\end{align}
%Note that $N_2(m,n)$ is the same as $N(m,n)$ for $n\ge 1$. 
%\begin{align}\label{sptpng}
%Spt_j(n)=np(n)-\frac{1}{2}{}_jN_2(n),
%\end{align}
%where 
%\begin{equation}\label{3rm}
%{}_kN_{j}(n):=\sum_{m=-\infty}^{\infty}m^jN_{k}(m,n).
%\end{equation}
%Here, 
Dyson \cite{dyson} proved that for $n>1$, 
\begin{align}\label{fdyson}
np(n)=\frac{1}{2} M_2(n),
\end{align}
where $M_2(n)$ is defined in (\ref{crankm}). Thus, in \cite{andrews}, Andrews indeed studied the difference of the second moments of crank and rank.
Inspired by Andrews' results, Garvan \cite{garvan} investigated a further relationship by studying the difference of the $2k$-th symmetrized moments of rank and crank. He considered the higher order smallest part function $\textup{spt}_k(n)$, that specializes to $\textup{spt}(n)$ for $k=1$, and he discovered many interesting arithmetic properties of $\textup{spt}_k(n)$. Garvan first defined the symmetrized crank moment $\mu_{k}(n)$ by
\begin{equation}\label{1symj}
\mu_{k}(n)=\sum_{m=-\infty}^{\infty}\binom{m+\lfloor \frac{k-1}{2}\rfloor}{k}M(m,n)
\end{equation}
and showed that
\begin{align*}
\sum_{n=1}^{\infty} \mu_{2k}(n)q^n &=\frac{1}{(2k)!} \left.\left(\left(\frac{d}{dz}\right)^{2k} z^{k-1} C(z,q)\right) \right |_{z=1},
\end{align*}
where $C(z,q)$ is defined in (\ref{crankgf}).
%\begin{align*}
%C(z,q)=\sum_{n=0}^{\infty} \sum_{m=-n}^{n} M(m,n) z^m q^n=\frac{(q)_{\infty}}{(zq)_{\infty} (z^{-1}q)_{\infty}}.
%\end{align*}
From  \cite{andrews1}, we have
\begin{align*}
\sum_{n=1}^{\infty} \eta_{2k}(n) q^n = \frac{1}{(2k)!} \left. \left(\left(\frac{d}{dz}\right)^{2k} z^{k-1} R(z,q)\right) \right |_{z=1},
\end{align*}
where $R(z,q)$ is defined in (\ref{rankgf}) and
\begin{equation}\label{2symj}
\eta_{k}(n)=\sum_{m=-\infty}^{\infty}\binom{m+\lfloor \frac{k-1}{2}\rfloor}{k}N(m,n).
\end{equation}
%\begin{align*}
%R(z,q)=\sum_{n=0}^{\infty} \sum_{m=-n}^{n} N(m,n) z^m q^n & =1+ \sum_{n=1}^{\infty} \frac{q^{n^2}}{(zq)_n (z^{-1}q)_n}.
%\sum_{n=0}^{\infty} N(m,n)q^n &= \frac{1}{(q)_{\infty}} \sum_{n=1}^{\infty} (-1)^{n-1} q^{n(3n-1)/2+|m|n } (1-q^n).  
%\end{align*}
He \cite{garvan} then defined the higher order smallest part function $\textup{spt}_k(n)$ as
\begin{align*}
\textup{spt}_k(n)=\mu_{2k}(n)-\eta_{2k}(n),
\end{align*}
and proved that
\begin{align}\label{garvansptk}
\sum_{n=1}^{\infty}\textup{spt}_k(n)q^n=\sum_{n_k\ge \cdots \ge  n_1\ge 1} \frac{q^{n_1+\cdots+n_{k}}} {(1-q^{n_k})^2\cdots (1-q^{n_1})^2(q^{n_1+1};q)_{\infty}}.
\end{align}

%In 1994 \cite{garvan_successive},  Garvan generalized the partition rank to $k$-rank by introducing successive Durfee squares.   Let $N_{k}(m,n)$ be the number of partitions of $n$ into at least $k-1$ successive Durfee squares with $k$-rank equal to $m$.   Then, he showed that for $k\ge 2$, 
%\begin{align}
%\sum_{n=0}^{\infty} N_{k}(m,n)q^n =\frac{1}{(q)_{\infty}} \sum_{n=1}^{\infty} (-1)^{n-1} q^{n((2k-1)n-1)/2+|m|n}(1-q^n). \label{Nk}
%\end{align}
%and
%Note that $N_2(m,n)$ is the rank function $N(m,n)$ for $n\ge 1$. 

%Let $M(m,n)$ counts the number of partitions of $n$ with crank $m$. The crank function $M(m,n)$ has the generating function that is very analogous to \eqref{Nk}, namely 
%\begin{align}
%\sum_{n=0}^{\infty} M(m,n) q^n & =\frac{1}{(q)_{\infty}} \sum_{n=1}^{\infty} (-1)^{n-1} q^{n(n-1)/2+ |m|n}(1-q^n). \label{crank}
%\end{align}
%By comparing \eqref{Nk} and \eqref{crank},  we see that for $n\ge 1$
%\begin{align*}
%N_1(m,n)=M(m,n).
%\end{align*}

%F.~Dyson \cite{dyson} proved that for $n>1$, 
%\begin{align*}
%np(n)=\frac{1}{2} M_2(n),
%\end{align*}
%where $M_2(n)=\sum_{m=-n}^{n}  m^2 M(m,n)$.  Thus, 
%in \cite{andrews}, Andrews indeed studied the difference of the second moments of $N_1(m,n)$ and $N_2(m,n)$, and in \cite{garvan}, Garvan investigated further relationship by taking their $2k$-th symmetrized moments. 
If we define the $k$-th symmetrized $j$-rank function by
\begin{equation}\label{ksymj}
{}_j\mu_{k}(n):=\sum_{m=-\infty}^{\infty}\binom{m+\lfloor \frac{k-1}{2}\rfloor}{k}N_{j}(m,n),
\end{equation}
then it is easy to see that ${}_1\mu_{k}(n)=\mu_{k}(n)$ and ${}_2\mu_{k}(n)=\eta_k(n)$. Therefore, a natural question is to see if it is possible to generalize the work of Andrews and Garvan using the $2k$-th symmetrized moments of $j$-rank. We do this here by generalizing $\textup{spt}_k(n)$ to ${}_j\textup{spt}_k(n)$. When $k=1$, we show how ${}_j\textup{spt}_1(n)$ can be represented in terms of $\textup{Spt}_j(n)$.

Garvan \cite{garvan} proved that $M_{2k}(n)>N_{2k}(n)$ for all $k\geq 1$ and $n\geq 1$. To prove this inequality, he used an analogue of Stirling numbers of the second kind, namely $S^{*}(n,k)$, to relate the ordinary and symmetrized moments. The numbers $S^{*}(n,k)$ are defined by \cite{garvan}
\begin{equation*}
x^{2n}=\sum_{k=1}^{n}S^{*}(n,k)g_{k}(x),
\end{equation*}
for $n\geq 1$, where for $k\geq 1$,
\begin{equation*}
g_k(x)=\prod_{j=0}^{k-1}(x^2-j^2).
\end{equation*}
The above inequality between the rank and crank moments can be easily generalized to the following inequality between moments of $j$-rank and $(j+1)$-rank.
\begin{theorem}\label{genineq}
For all $j, k ,n\geq 1$, let ${}_jN_{k}(n)$ be defined in \textup{(\ref{3rm})}. Then,
\begin{equation}
{}_jN_{2k}(n)>{}_{j+1}N_{2k}(n).
\end{equation}
\end{theorem}
This paper is organized as follows. In Section 2, we prove (\ref{sptpng}). Then in Section 3, we give a combinatorial interpretation of $\textup{Spt}_j(n)$ and explain the motivation behind generalizing Garvan's $\textup{spt}_k(n)$ to ${}_j\textup{spt}_{k}(n)$ by studying the difference of two $2k$-th symmetrized $j$-rank functions. In Section 4, we prove some lemmas involving the $k$-th symmetrized $j$-rank function and obtain the generating function of ${}_j\textup{spt}_{k}(n)$. In Section 5, we give a combinatorial interpretation of ${}_j\textup{spt}_{k}(n)$. Finally, in Section 6, we prove Theorem \ref{genineq}.

\section{Proof of (\ref{sptpng})}

We first prove the following result.
\begin{theorem}\label{genn}
We have
\begin{align}\label{genn1}
&\sum_{n_{j}\geq 1}\sum_{n_{j-1}\geq\cdots\geq n_{1}\geq 0}
\frac{q^{n_1^2+\cdots +n_{j-1}^2+ n_{j}} (q)_{n_{j}} }{(q)_{n_1} (q)_{n_2-n_1}\cdots (q)_{n_j-n_{j-1}} (1-q^{n_{j}})^2(q^{n_{j}+1})_{\infty}} \nonumber\\
&=\frac{1}{(q)_{\infty}}\sum_{n=1}^{\infty}\frac{nq^n}{1-q^n}+\frac{1}{(q)_{\infty}}\sum_{n=1}^{\infty}\frac{(-1)^nq^{n((2j+1)n+1)/2}(1+q^n)}{(1-q^n)^2}.
\end{align}
\iffalse
\begin{align}\label{genn1}
&\sum_{n_{j}\geq 1}\sum_{n_{j-1}\geq\cdots\geq n_{1}\geq 0}
\frac{q^{n_{j}} (q)_{n_{j}} }{(1-q^{n_{j}})^2(q^{n_{j}+1})_{\infty}}\left[{n_{j}\atop n_{j-1}}\right]\cdots \left[{n_{2}\atop n_{1}}\right]q^{n_{1}^{2}+\cdots+n_{j-1}^{2}}\nonumber\\
&=\frac{1}{(q)_{\infty}}\sum_{n=1}^{\infty}\frac{nq^n}{1-q^n}+\frac{1}{(q)_{\infty}}\sum_{n=1}^{\infty}\frac{(-1)^nq^{n((2j+1)n+1)/2}(1+q^n)}{(1-q^n)^2}.
\end{align}
\fi
\end{theorem}

We first need the following two lemmas.
\begin{lemma}\label{lmn}
We have
\begin{align}\label{kn1}
&\sum_{n_{j}\geq n_{j-1}\geq\cdots\geq n_{1}\geq 0}\frac{(z)_{n_{j}}(z^{-1})_{n_{j}}q^{n_{1}^{2}+\cdots+n_{j-1}^{2}+n_{j}}}{(q)_{n_{1}}(q)_{n_{2}-n_{1}}\cdots (q)_{n_{j}-n_{j-1}}}\nonumber\\
&=\frac{(zq)_{\infty}(z^{-1}q)_{\infty}}{(q)_{\infty}^{2}}\left(1+\sum_{n=1}^{\infty}\frac{(-1)^nq^{n((2j+1)n+1)/2}(1+q^n)(z;q)_{n}(z^{-1};q)_n}{(zq;q)_n(z^{-1}q;q)_n}\right).
\end{align}
\end{lemma}
\begin{proof}
Let $b_{j+1}=z=c_{j+1}^{-1}$, $b_{1}, c_{1}, b_{2}, c_{2},\cdots b_{j}, c_{j}\to\infty, N\to\infty, a\to 1$ in (\ref{andrewswatsong}). This gives (\ref{kn1}) upon simplification.
\end{proof}

\begin{lemma}\label{d2rk}
We have
\begin{equation}
\left.\frac{d^2}{dz^2}R_{j}(z,q)\right|_{z=1}=\frac{-2}{(q)_{\infty}}\sum_{n=1}^{\infty}\frac{(-1)^nq^{n((2j-1)n+1)/2}(1+q^n)}{(1-q^n)^2}.
\end{equation}
\end{lemma}
\begin{proof}
From (\ref{R_k}), we have
\begin{equation}
R_{j}(z,q)=\frac{z}{(q)_{\infty}} \sum_{n=-\infty \atop n\neq 0}^{\infty} (-1)^{n-1} q^{n((2j-1)n+1)/2} \frac{1-q^n}{1-zq^n}.
\end{equation}
Differentiating both sides with respect to $z$, we have
\begin{equation}
\frac{d^2}{dz^2}R_{j}(z,q)=\frac{-2}{(q)_{\infty}} \sum_{n=-\infty \atop n\neq 0}^{\infty} (-1)^{n} q^{n((2j-1)n+1)/2+n} \frac{1-q^n}{(1-zq^n)^3}.
\end{equation}
Now let $z=1$ to see that
\begin{align}
\left.\frac{d^2}{dz^2}R_{j}(z,q)\right|_{z=1}&=\frac{-2}{(q)_{\infty}} \sum_{n=-\infty \atop n\neq 0}^{\infty} \frac{(-1)^{n} q^{n((2j-1)n+1)/2+n}} {(1-q^n)^2}\nonumber\\
&=\frac{-2}{(q)_{\infty}}\sum_{n=1}^{\infty}\frac{(-1)^nq^{n((2j-1)n+1)/2}(1+q^n)}{(1-q^n)^2}.
\end{align}
\end{proof}

\begin{proof}[Theorem \textup{\ref{genn}}][]
The idea is to take the second derivative on both sides of (\ref{kn1}) with respect to $z$ and then let $z=1$. Since \cite[Equation (2.1)]{andrews} 
\begin{equation}\label{f2d}
\left[\frac{d^{2}}{dz^{2}}(1-z)(1-z^{-1})f(z)\right]_{z=1}=-2f(1),
\end{equation}
we see that
\begin{align}\label{dlhs}
&\left[\frac{d^2}{dz^2}\sum_{n_{j}\geq n_{j-1}\geq\cdots\geq n_{1}\geq 0}\frac{(z)_{n_{j}}(z^{-1})_{n_{j}}q^{n_{1}^{2}+\cdots+n_{j-1}^{2}+n_{j}}}{(q)_{n_{1}}(q)_{n_{2}-n_{1}}\cdots (q)_{n_{j}-n_{j-1}}}\right]_{z=1}\nonumber\\
&=-2\sum_{n_{j}\geq 1}\sum_{n_{j-1}\geq\cdots\geq n_{1}\geq 0}\frac{(q)_{n_{j}-1}^{2}q^{n_{1}^{2}+\cdots+n_{j-1}^{2}+n_{j}}}{(q)_{n_{1}}(q)_{n_{2}-n_{1}}\cdots(q)_{n_{j}-n_{j-1}}}. \nonumber\\
%&=-2\sum_{n_{j}\geq 1}\sum_{n_{j-1}\geq\cdots\geq n_{1}\geq 0}\frac{(q)_{n_{j}}q^{n_{j}}}{(1-q^{n_{j}})^2}\left[{n_{j}\atop n_{j-1}}\right]\cdots \left[{n_{2}\atop n_{1}}\right]q^{n_{1}^{2}+\cdots+n_{j-1}^{2}}.\nonumber\\
\end{align}
From \cite[Equation (2.4)]{andrews}, we have
\begin{equation}\label{eis2}
\left[\frac{d^2}{dz^2}\frac{(zq)_{\infty}(z^{-1}q)_{\infty}}{(q)_{\infty}^2}\right]_{z=1}=-2\sum_{n=1}^{\infty}\frac{nq^{n}}{1-q^n}.
\end{equation}
Now
\begin{equation}\label{der3}
\left[\frac{d}{dz}\left(1+\sum_{n=1}^{\infty}\frac{(-1)^nq^{n((2j-1)n+1)/2}(1+q^n)(z)_{n}(z^{-1})_n}{(zq)_n(z^{-1}q)_n}\right)\right]_{z=1}=0.
\end{equation}
Using (\ref{f2d}), we have
\begin{align}\label{der4}
&\left[\frac{d^2}{dz^2}\left(1+\sum_{n=1}^{\infty}\frac{(-1)^nq^{n((2j-1)n+1)/2}(1+q^n)(z)_{n}(z^{-1})_n}{(zq)_n(z^{-1}q)_n}\right)\right]_{z=1}\nonumber\\
&=-2\sum_{n=1}^{\infty}\frac{(-1)^nq^{n((2j-1)n+1)/2}(1+q^n)}{(1-q^n)^2}.
\end{align}
Then from (\ref{kn1}), (\ref{dlhs}), (\ref{eis2}), (\ref{der3}) and (\ref{der4}), we obtain (\ref{genn1}) upon simplification. This completes the proof.
\end{proof}

Now define $\textup{Spt}_j(n)$ by
\begin{equation}\label{sptjn}
\sum_{n=1}^{\infty}\textup{Spt}_j(n)q^n:=\sum_{n_{j}\geq 1}\sum_{n_{j-1}\geq\cdots\geq n_{1}\geq 0}\frac{q^{n_{j}}}{(1-q^{n_{j}})^2(q^{n_{j}+1})_{\infty}}\left[{n_{j}\atop n_{j-1}}\right]\cdots \left[{n_{2}\atop n_{1}}\right]q^{n_{1}^{2}+\cdots+n_{j-1}^{2}},
\end{equation}
where
\begin{align*}
\left[\begin{matrix} n\\m \end{matrix}\right]=\begin{cases} \frac{(q)_n}{(q)_{m} (q)_{n-m}}, & \text{ if $0\le m\le n$},\\
0, & \text{ otherwise.}
\end{cases}
\end{align*}
From \cite[Equation (3.3)]{andrews}, we have
\begin{equation}\label{gfnpn}
\sum_{n=1}^{\infty}np(n)q^n=\frac{1}{(q)_{\infty}}\sum_{n=1}^{\infty}\frac{nq^n}{1-q^n}.
\end{equation}
Also, from (\ref{R_k}) and the fact that the odd moments of $j$-rank are equal to zero, we have 
\begin{align}
\left.\frac{d^2}{dz^2}R_{j}(z,q)\right|_{z=1}&=\sum_{n=1}^{\infty}\sum_{m=-\infty}^{\infty}m(m-1)N_{j}(m,n)q^n\nonumber\\
&=\sum_{n=1}^{\infty}\sum_{m=-\infty}^{\infty}m^2N_{j}(m,n)q^n\nonumber\\
&=\sum_{n=1}^{\infty}{}_jN_{2}(n)q^n.
\end{align}
Along with Lemma \ref{d2rk}, this implies
\begin{equation}\label{jn2n}
-\frac{1}{2}\sum_{n=1}^{\infty}{}_jN_{2}(n)q^n=\frac{1}{(q)_{\infty}}\sum_{n=1}^{\infty}\frac{(-1)^nq^{n((2j-1)n+1)/2}(1+q^n)}{(1-q^n)^2}.
\end{equation}
Finally, Theorem \ref{genn} along with (\ref{sptjn}), (\ref{gfnpn}) and (\ref{jn2n}) gives (\ref{sptpng}). This completes the proof.

\textbf{Remarks.} 1. If $j>n$, then ${}_{j+1}N_{2}(n)=0$ as $N_{j+1}(m,n)=0$. Then (\ref{sptpng}) implies that 
\begin{equation}\label{jgn}
\textup{Spt}_j(n)=np(n),\hspace{1mm}\textup{for}\hspace{1mm} j>n.
\end{equation}
This, along with (\ref{fdyson}) gives
\begin{equation}
\textup{Spt}_{\infty}(n)=np(n)=\frac{1}{2}M_{2}(n),
\end{equation}
so that we have by using (\ref{sptpng}), 
\begin{equation}
\textup{Spt}_{\infty}(n)-\textup{Spt}_{1}(n)=\frac{1}{2}N_{2}(n).
\end{equation}
Also, in view of (\ref{pc5})-(\ref{pc7}), we see that, 
\begin{equation}
\textup{Spt}_j(\ell n+m)\equiv 0\hspace{1mm}(\textup{mod}\hspace{1mm} \ell)\hspace{1mm}\textup{for}\hspace{1mm}(\ell,m)=(5,4), (7,5)\hspace{1mm}\textup{and}\hspace{1mm}(11,6)\hspace{1mm}\textup{and}\hspace{1mm}j>\ell n+m.
\end{equation}
%Since $|q|<1$, if we let $j\to\infty$ in (\ref{jn2n}), we see that $\lim_{j\to\infty}{}_jN_{2}(n)=0$. Combining with (\ref{sptpng}), we see that
%\begin{equation}\label{sptjinf}
%\lim_{j\to\infty}\textup{Spt}_j(n)=np(n).
%\end{equation}
2. Note that from (\ref{sptpng}), we have
\begin{equation}\label{sptdiff}
\textup{Spt}_{j}(n)-\textup{Spt}_{j-1}(n)=\frac{1}{2}\left({}_jN_{2}(n)-{}_{j+1}N_{2}(n)\right).
\end{equation}

\section{A combinatorial interpretation of $\textup{Spt}_j(n)$}

Here we give a combinatorial interpretation of $\textup{Spt}_j(n)$ defined in (\ref{sptjn}). 

For a partition $\pi$, we take the largest square that fits inside the Ferrers digram of $\pi$ starting from the lower left corner. We call this square the lower-Durfee square.  The partition $\pi$ can be divided into two portions: the square and the parts to its right, and the parts above the square. 
%We now successively define lower-Durfee squares. 
If there exists one, we take a second lower-Durfee square that fits inside $\pi$ right above the first lower-Durfee square. %Clearly, if there are at least as many parts as the side of the lower-Durfee square, the second lower-Durfee square exists.  
We can successively define lower-Durfee squares as long as there exist parts in the upper portion.  For an $s\ge 0$, we call these $s$ lower-Durfee squares defined from the bottom the $s$ successive lower-Durfee squares of $\pi$.

%For a $j\ge 1$, we define $spt_{j}(n)$ as follows. 

Throughout this paper,  if a positive integer occurs as a part in a partition, we mark all of its occurences with positive integers in an increasing order from the left to right. For instance, for $\pi=5+5+4+3+3+3$, we write $\pi=5_1+5_2+4_1+3_1+3_2+3_3$ and call the subscript of each part of that partition as its \emph{mark}.

Take a partition $\pi$ of $n$ and consider its $j-1$ successive lower-Durfee squares, and define a weight of $\pi$ by
\begin{align}\label{wjpi}
W_{j}(\pi)=\sum_{i_m} m,
\end{align} 
where the sum is over the part right above the $(j-1)$st lower-Durfee square if it exists and all the parts  that are contained in the $j-1$ successive lower-Durfee squares. For instance, let $j=3$ and %$\pi=9+8+8+8+8+6+6+5+4+4+3$. With this notion, the partition $\pi$ can be written as 
$\pi=9_1+8_1+8_2+8_3+8_4+6_1+6_2+5_1+4_1+4_2+3_1$. Consider its first $2$ successive lower-Durfee squares of sides $3$ and $5$ as shown in Figure 1. Then, 
\begin{align*}
W_{3}(\pi)=2+3+4+1+2+1+1+2+1= 17.
\end{align*}
For $\pi=4_1+4_2+3_1+3_2+2_1$, we have
\begin{align*}
W_3(\pi)=1+2+1+2+1=7. 
\end{align*}
If $\pi$ has fewer than $j-1$ successive lower-Durfee squares, we define 
\begin{align*}
W_{j}(\pi)=\sum_{i_m} m,
\end{align*}
where the sum is over all the parts of $\pi$.  For instance, if $\pi=4_1+4_2$, then
\begin{align*}
W_3(\pi)= 1+2=3. 
\end{align*}
We show that the following theorem holds.
{\allowdisplaybreaks\begin{center}
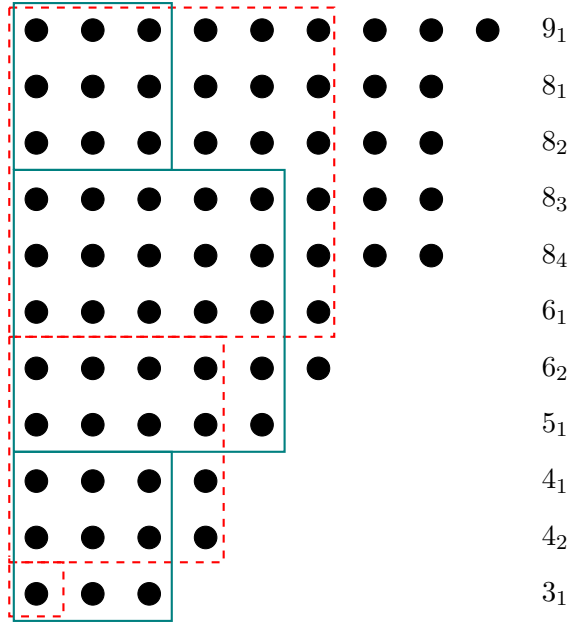
\begin{figure}
\begin{tikzpicture}[scale=3]
%\draw (.5,.5) node {$\pi: 9 + 8 + 8 + 8 + 8 + 6 + 6 + 5 + 4 + 4 + 3$};
%\draw (-.5, -.5) node {$\lambda = $ };
\foreach \x in {0,...,8}
	\filldraw (\x*.25, 0) circle (.5mm);
\foreach \x in {0,...,7}
	\filldraw (\x*.25, -.25) circle (.5mm);
	\foreach \x in {0,...,7}
	\filldraw (\x*.25, -.5) circle (.5mm);
	\foreach \x in {0,...,7}
	\filldraw (\x*.25, -.75) circle (.5mm);
	\foreach \x in {0,...,7}
	\filldraw (\x*.25, -1) circle (.5mm);
\foreach \x in {0,...,5}
	\filldraw (\x*.25, -1.25) circle (.5mm);
\foreach \x in {0,...,5}
	\filldraw (\x*.25, -1.5) circle (.5mm);
\foreach \x in {0,...,4}
	\filldraw (\x*.25, -1.75) circle (.5mm);
\foreach \x in {0,...,3}
	\filldraw (\x*.25, -2) circle (.5mm);
	\foreach \x in {0,...,3}
	\filldraw (\x*.25, -2.25) circle (.5mm);
	\foreach \x in {0,...,2}
	\filldraw (\x*.25, -2.5) circle (.5mm);
%\draw [->] (-0.1,-0.62) -- (-0.1,-2);	
\draw[dashed, thick, red] (-0.12, 0.1) -- (1.32, 0.1) -- (1.32, -1.36) -- (-0.12, -1.36) -- (-0.12, 0.1);
\draw[dashed, thick, red] (-0.12, -1.36) -- (0.83, -1.36) -- (0.83, -2.36) -- (-0.12, -2.36) -- (-0.12, -1.33);
\draw[dashed, thick, red] (-0.12, -2.36) -- (0.12, -2.36) -- (0.12, -2.60) -- (-0.12, -2.60) -- (-0.12, -2.33);
\draw[thick, teal] (-.1,-1.87) -- (.6,-1.87) -- (.6,-2.62) -- (-0.1, -2.62) -- (-0.1,-1.87);	
\draw[thick, teal] (-.1,-1.87) -- (1.1,-1.87) -- (1.1,-0.62) -- (-0.1,-0.62) -- (-0.1,-1.87);	
\draw[thick, teal] (-.1,-0.62) -- (0.6, -0.62) -- (0.6, 0.12) -- (-.1, 0.12) -- (-0.1, -0.62);
\draw (2.3, 0) node {$9_{1}$};
\draw (2.3, -0.25) node {$8_{1}$};
\draw (2.3, -0.5) node {$8_{2}$};
\draw (2.3, -0.75) node {$8_{3}$};
\draw (2.3, -1.) node {$8_{4}$};
\draw (2.3, -1.25) node {$6_{1}$};
\draw (2.3, -1.5) node {$6_{2}$};
\draw (2.3, -1.75) node {$5_{1}$};
\draw (2.3, -2) node {$4_{1}$};
\draw (2.3, -2.25) node {$4_{2}$};
\draw (2.3, -2.5) node {$3_{1}$};

%\draw[ultra thin, teal] (-.1,.1) -- (-.1,-.6) -- (.6, -.6);
%\draw[ultra thin, teal] (-.1,-.6) -- (-.1, -1.1) -- (.35, -1.1) -- (0.35,-.6) -- (-.1,-.6);
\end{tikzpicture}
\caption{$\pi: 9_1+8_1+8_2+8_3+8_4+6_1+6_2+5_1+4_1+4_2+3_1$.}
\end{figure}
\end{center}}
%\textup{A partition} \pi \textup{of} 63
\begin{theorem}\label{sptwkpi}
With $\textup{Spt}_j(n)$ and $W_j(\pi)$ defined in \textup{(\ref{sptjn})} and \textup{(\ref{wjpi})} respectively, we have
\begin{align*}
\textup{Spt}_j(n)=\sum_{\pi} W_j(\pi),
\end{align*}
where the sum is over all partitions of $n$.
\end{theorem}
Before we proceed, we make some remarks on lower-Durfee squares. Suppose that $\pi$ has exactly $s$ successive lower-Durfee squares of sides $d_1, d_2, \ldots, d_s$ from the bottom to top. Then 
\begin{align}
d_i\le d_{i+1}
\end{align}
for $i=1,\ldots, s-2$. However, it is not necessary that $d_{s-1}\le d_s$ since $d_s<d_{s-1}$ if there exist less than $d_{s-1}$ parts above the $s-1$st lower-Durfee square.  Also, for $i=1,\ldots, s-1$, all the parts below the $i$th lower-Durfee square cannot exceed $d_i$. If all parts below the $s$-th lower-Durfee square are less than or equal to $d_s$, we call $\pi$ a \emph{Rogers-Ramanujan partition with $s$ successive lower-Durfee squares}. 

The dotted lines in Figure 1 form the successive Durfee squares of the partition whereas the complete lines form the successive lower-Durfee squares.

We now prove two lemmas which are crucial for the proof of Theorem \ref{sptwkpi}.  

\begin{lemma} \label{lemma1}
Let $\pi$ be a Rogers-Ramanujan partition with $s$ successive lower-Durfee squares. Then $\pi$ is a partition with exactly $s$ successive Durfee squares. Indeed, the lower-Durfee squares form the Durfee squares. 
\end{lemma}

\proof
We prove by induction on $s$. Clearly, the statement holds true for $s=1$.  

For $s>1$, we now show that the $s$-th lower-Durfee square is the Durfee square of $\pi$.  By construction, since $\pi$ is a Rogers-Ramanujan partition, the part right below the $s$-th lower-Durfee square is less than or equal to $d_{s}$, so there are exactly $d_{s}$ parts greater than or equal to $d_{s}$. Thus, the first Durfee square of $\pi$ has to be of side $d_{s}$, namely the first Durfee square matches the $s$-th lower-Durfee square. The parts below the first Durfee square form a Rogers-Ramanujan partition with $s-1$ successive lower-Durfee squares. By induction hypothesis, it follows that the lower-Durfee squares form the Durfee squares. This completes the proof. 
\endproof

\begin{lemma}\label{lemma3}
Let $\pi$ be a partition with exactly $s$ successive lower-Durfee squares.  Then $\pi$ has exactly $s$ successive Durfee squares.
\end{lemma}

\proof
We prove by induction on $s$.  For $s=1$, clearly $\pi$ is a Rogers-Ramanujan partition with one lower-Durfee square. Thus, the statement follows from Lemma~\ref{lemma1}.  

For $s>1$,  if $\pi$ is a Rogers-Ramanujan partition, then it follows from Lemma~\ref{lemma1}. Otherwise, we now show that the side of the first Durfee square of $\pi$ is less than $d_s+d_{s-1}$.  From construction of successive lower-Durfee squares, the smallest part in the $s-1$st lower-Durfee square is equal to its side $d_{s-1}$. Thus, the square of side $d_s+d_{s-1}$ cannot fit inside $\pi$, which implies that the parts below the first Durfee square of $\pi$ form a partition with $s-1$ successive lower-Durfee squares. It follows from the induction hypothesis that the partition with $s-1$ successive lower-Durfee squares has exactly $s-1$ successive Durfee squares. Therefore, $\pi$ has exactly $s$ successive Durfee squares. 
\endproof

%We are now ready to give a proof of (\ref{sptwkpi}).
\begin{proof}[Theorem \textup{\ref{sptwkpi}}][]
Consider the series on the right-hand side of (\ref{sptjn}). Since, for $n_j \ge 1$, 
\begin{align*}
\frac{q^{n_j}}{(1-q^{n_j})^2(q^{n_j+1})_{\infty}}=\frac{q^{n_j}+2q^{2n_j}+3q^{3n_j}+\cdots}{(q^{n_j+1})_{\infty}},
\end{align*}
the outer summation generates partitions $\nu$ with the smallest part equal to $n_{j}$ and weight equal to the number of occurrences of $n_{j}$.  Also, note that
\begin{align*}
\sum_{n_{j-1}\geq\cdots\geq n_{1}\geq 0}  \left[{n_{j}\atop n_{j-1}}\right]\cdots \left[{n_{2}\atop n_{1}}\right]q^{n_{1}^{2}+\cdots+n_{j-1}^{2}}
\end{align*}
generates Rogers-Ramanujan partitions $\mu$ with the largest part $\le n_{j}$ and at most $j-1$ successive Durfee squares. Thus, the union of the parts of $\mu$ and $\nu$ is a partition where the parts below the part $n_{j}$ form a Rogers-Ramanujan partition with at most $j-1$ successive Durfee squares. 

For a partition $\pi$ of $n$, we take successive lower-Durfee squares, whose sides are $d_1, d_2, \ldots$ from the bottom to top. If there are only $s$ lower-Durfee squares, we define $d_{i}=0$ for $i>s$.  For convenience, we write the parts of $\pi$ in increasing order, namely $\pi_1$ is the smallest, $\pi_2$ is the second smallest, etc.  For $i=0,\ldots, d$, $d=d_1+d_2+\cdots+d_{j-1}$, we define a pair of partition $\mu^i$ and $\nu^i$ by
\begin{align*}
\mu^i&=\pi_1+\cdots +\pi_i,\\
\nu^i&=\pi_{i+1}+\cdots. 
\end{align*}
From the construction, $\mu^i$ has at most $j-1$ successive lower-Durfee squares. Thus it follows from Lemma~\ref{lemma3} that $\mu^i$ has at most $j-1$ successive Durfee squares.  In addition,
we see that $\mu^i$ and $\nu^i$ for $i=0,\ldots, d$ are the only possible pairs for $\mu$ and $\nu$ generated by the right hand of \eqref{sptjn} which make $\pi$.  

\iffalse
If $\pi$ has less than $k-2$ successive lower-Durfee squares, by adding the largest part of $\pi$ as many copies as needed to make it have exactly $k-2$ successive lower-Durfee squares, we can come back to the case above with the resulting partition. Then $\pi$ is $\mu^m$ and the added parts form $\nu^m$ for some $m$.  Then it follows that $\mu^{m}$, namely $\pi$ cannot have more than $k-2$ successive Durfee squares and $\mu^i$ and $\nu^i$ for $i\le m$ are the only possible pairs that make $\pi$. 
\fi

Each pair $\mu^i$ and $\nu^i$ is counted with weight equal to the number of appearances of the smallest part of $\nu^i$ in $\nu^i$, namely $\pi_{i+1}$. By marking the same parts in increasing order as introduced at the beginning of this section, we see that the number of appearances of $\pi_{i+1}$ is its mark. So, the partition $\pi$ is generated by the right hand side of \eqref{sptjn} with weight equal to the sum of the marks of $\pi_{1}$ through $\pi_{d+1}$, which is exactly the same as $W_j(\pi)$. Therefore, the coefficient of $q^n$ on the left-hand side of (\ref{sptjn}) is equal to $\sum_{\pi} W_j (\pi)$. This completes the proof.
%\endproof
\end{proof}
\textbf{Remarks.} 1. When $j=1$, we take a partition $\pi$ of $n$ and consider $j-1=0$ successive lower Durfee squares. Thus, $W_1(\pi)$ is nothing but the number of appearances of the smallest part of $\pi$. Hence, $\textup{Spt}_1(n)=\textup{spt}(n)$.\\

2. By letting $j$ go to infinity in Theorem~\ref{sptwkpi}, we see that $\text{Spt}_{\infty}(n)$ counts the sum of the marks of the parts of all the partitions of $n$. \\

3. From (\ref{sptdiff}) and the fact that the odd moments of $j$-rank are equal to zero, we have
\allowdisplaybreaks{\begin{align}\label{sptmot}
\textup{Spt}_{j}(n)-\textup{Spt}_{j-1}(n)&=\frac{1}{2}\left({}_jN_{2}(n)-{}_{j+1}N_{2}(n)\right)\nonumber\\
&=\frac{1}{2}\sum_{m=-\infty}^{\infty}m^2\left(N_{j}(m,n)-N_{j+1}(m,n)\right)\nonumber\\
&=\frac{1}{2}\sum_{m=-\infty}^{\infty}(m^2-m)\left(N_{j}(m,n)-N_{j+1}(m,n)\right)\nonumber\\
&={}_j\mu_2(n)-{}_{j+1}\mu_2(n),
\end{align}}
where ${}_j\mu_{k}(n)$ is defined in (\ref{ksymj}). When $j=1$, we have seen that this gives nothing but $\textup{spt}(n)$ since by (\ref{fdyson}), we have $\textup{Spt}_0(n)=0$. In light of what Garvan has done for his higher-order spt-function, this gives us a motivation to study the difference ${}_j\mu_{2k}(n)-{}_{j+1}\mu_{2k}(n)$. We make the following definition:
\begin{definition} 
For $j,k \ge 1$, define
\begin{align}
{_j}\textup{spt}_k(n)  = {_j}\mu_{2k} (n)-{_{j+1}}\mu_{2k} (n).  \label{def3.4}
\end{align}
\end{definition}
We call ${_j}\textup{spt}_k(n)$ a \emph{generalized higher order spt-function}.

\section{Generating function for the generalized higher order spt-function ${}_j\textup{spt}_{k}(n)$}
We begin with some lemmas involving the $2k$-th symmetrized $j$-rank function which will be used in the sequel.

%\begin{definition}
%Define the $k$-th symmetrized $j$-rank function by
%\begin{equation}\label{ksymj}
%{}_j\mu_{k}(n)=\sum_{m=-\infty}^{\infty}\binom{m+\lfloor \frac{k-1}{2}\rfloor}{k}N_{j}(m,n).
%\end{equation}
%\end{definition}

\begin{lemma} \label{lemma2.2}
For $j, k\ge 1$, we have
\begin{align*}
\sum_{n=1}^{\infty}{}_j\mu_{2k}(n)q^n = \frac{1}{(2k)!}  \left. \left(\left(\frac{d}{dz}\right)^{2k} z^{k-1} R_j(z,q)\right) \right |_{z=1}.
\end{align*}
\end{lemma}

\proof
Since
\begin{equation*}
z^{k-1}R_{j}(z,q)=\sum_{m=-\infty}^{\infty}\sum_{n=1}^{\infty}N_{j}(m,n)z^{m+k-1}q^n,
\end{equation*}
we have
\begin{align*}
\frac{1}{(2k)!} \left. \left(\left(\frac{d}{dz}\right)^{2k} z^{k-1} R_j(z,q)\right) \right |_{z=1}&=\frac{1}{(2k)!}\sum_{m=-\infty}^{\infty}\sum_{n=1}^{\infty}(m+k-1) %(m+k-2)
\cdots (m-k)N_{j}(m,n)q^n\\ 
&=\frac{1}{(2k)!}\sum_{n=1}^{\infty}\sum_{m=-\infty}^{\infty}\frac{(m+k-1)!}{(m-k-1)!}N_{j}(m,n)q^n\\
&=\sum_{n=1}^{\infty}\sum_{m=-\infty}^{\infty}\binom{m+k-1}{2k}N_{j}(m,n)q^n \\
&=\sum_{n=1}^{\infty}{}_j\mu_{2k}(n)q^n,
\end{align*}
which completes the proof. 
\endproof

\begin{lemma}
For $j,k \ge 1$, we have
\begin{equation}\label{genjmu2k}
\sum_{n=1}^{\infty} {_j}\mu_{2k} (n) q^{n}=\frac{1}{(q)_{\infty}}\sum_{n=-\infty \atop n\neq 0}^{\infty}\frac{(-1)^{n-1} q^{n((2j-1)n+1)/2 +kn}}{(1-q^n)^{2k}}.
\end{equation}
\end{lemma}

\proof
By Leibniz's rule,
\begin{align*}
\frac{1}{(2k)!} \left.\left(\left(\frac{d}{dz}\right)^{2k} z^{k-1} R_j(z,q)\right)\right|_{z=1}=\frac{1}{(2k)!} \sum_{m=0}^{k-1} \binom{2k}{m} (k-1)\cdots (k-m) R_j^{(2k-m)}(1,q),
\end{align*}
and by \eqref{R_k},
\begin{align*}
R_j^{(m)}(z,q)=\frac{-m!}{(q)_{\infty}} \sum_{n=-\infty \atop n\neq 0}^{\infty}\frac{(-1)^n q^{n((2j-1)n+1)/2 +(m-1)n}(1-q^n)}{(1-zq^n)^{m+1}}.
\end{align*}
Hence, from Lemma~\ref{lemma2.2}, we see that
\iffalse
\begin{equation*}
\sum_{n=1}^{\infty} {_j}\mu_{2k} (n) q^{n} = \frac{1}{(2k)!}\left.\left(\left(\frac{d}{dz}\right)^{2k}\frac{z^k}{(q)_{\infty}} \sum_{n=-\infty \atop n\neq 0}^{\infty} (-1)^{n-1} q^{n(2j-1)n+1)/2} \frac{1-q^n}{1-zq^n} \right) \right |_{z=1}.
\end{equation*}
\fi
{\allowdisplaybreaks\begin{align*} %\label{jmu2kddz1}
&\sum_{n=1}^{\infty} {_j}\mu_{2k} (n) q^{n}\nonumber\\ &=\frac{1}{(2k)!}\left.\sum_{m=0}^{k-1}\binom{2k}{m}(k-1)\cdots(k-m)R_{j}^{(2k-m)}(z,q)\right|_{z=1}\nonumber\\
&=\frac{-1}{(2k)!(q)_{\infty}}\left.\sum_{m=0}^{k-1}(2k-m)!\binom{2k}{m}(k-1)\cdots(k-m)\sum_{n=-\infty \atop n\neq 0}^{\infty}\frac{(-1)^n q^{n((2j-1)n+1)/2 +(2k-m-1)n}(1-q^n)}{(1-zq^n)^{2k-m+1}}\right|_{z=1}\nonumber\\
&=\frac{1}{(q)_{\infty}}\sum_{n=-\infty \atop n\neq 0}^{\infty}\frac{(-1)^{n-1} q^{n((2j-1)n+1)/2 +(2k-1)n}}{(1-q^n)^{2k}}\sum_{m=0}^{k-1}\frac{(k-1)\cdots(k-m)}{m!}(q^{-n}-1)^{m}\nonumber\\
&=\frac{1}{(q)_{\infty}}\sum_{n=-\infty \atop n\neq 0}^{\infty}\frac{(-1)^{n-1} q^{n((2j-1)n+1)/2 +(2k-1)n}}{(1-q^n)^{2k}} \sum_{m=0}^{k-1} \binom{k-1}{m} (q^{-n}-1)^m \notag\\
&=\frac{1}{(q)_{\infty}}\sum_{n=-\infty \atop n\neq 0}^{\infty}\frac{(-1)^{n-1} q^{n((2j-1)n+1)/2 +(2k-1)n}}{(1-q^n)^{2k}} (1+q^{-n}-1)^{k-1} \notag\\
&=\frac{1}{(q)_{\infty}}\sum_{n=-\infty \atop n\neq 0}^{\infty}\frac{(-1)^{n-1} q^{n((2j-1)n+1)/2 + kn}}{(1-q^n)^{2k}}, \notag
\end{align*}}
\iffalse
where
\begin{equation*}
S(k,n,q):=\sum_{m=0}^{k-1}\frac{(k-1)\cdots(k-m)}{m!}q^{(k-m-1)n}(1-q^n)^{m}.
\end{equation*}
Now let $t=k-m$ so that
\begin{align}\label{sknq}
S(k,n,q)&=q^{-n}(1-q^n)^k\sum_{t=1}^{k}\frac{(k-1)(k-2)\cdots t}{(k-t)!}q^{tn}(1-q^n)^{-t}\nonumber\\
&=q^{-n}(1-q^n)^k\sum_{t=1}^{k}\binom{k-1}{t-1}q^{tn}(1-q^n)^{-t}\nonumber\\
&=(1-q^n)^{k-1}\sum_{t=0}^{k-1}\binom{k-1}{t}\left(\frac{q^n}{1-q^n}\right)^{t}\nonumber\\
&=(1-q^n)^{k-1}\left(1+\frac{q^n}{1-q^n}\right)^{k-1}\nonumber\\
&=1,
\end{align}
\fi
where in the penultimate step, we used the binomial theorem $(a+1)^{\ell}=\sum_{t=0}^{\ell}\binom{\ell}{t}a^{t}$.
\iffalse 
Hence, from (\ref{jmu2kddz1}) and (\ref{sknq}), we deduce that
\begin{equation}\label{genjmu2k}
\sum_{n=1}^{\infty} {_j}\mu_{2k} (n) q^{n}=\frac{1}{(q)_{\infty}}\sum_{n=-\infty \atop n\neq 0}^{\infty}\frac{(-1)^{n-1} q^{n((2j-1)n+1)/2 +kn}}{(1-q^n)^{2k}}.
\end{equation}
\fi
\endproof

We now need Garvan's theorem.
\begin{theorem}\cite[Theorem 3.3]{garvan}
Suppose $(\alpha_n,\beta_n)=(\alpha_n(1,q),\beta_n(1,q))$ is a Bailey pair with $a=1$ and $\alpha_0=1, \beta_0=1$.  Then
\begin{align*}
\sum_{n_k\ge \cdots \ge n_1\ge 1} \frac{(q)_{n_1}^2 q^{n_1+\cdots +n_{k}} \beta_{n_1}}{(1-q^{n_1})^2 \cdots (1-q^{n_k})^2}=\sum_{n_k\ge \cdots \ge n_1\ge 1} \frac{q^{n_1+\cdots +n_{k}}}{(1-q^{n_1})^2 \cdots (1-q^{n_k})^2} +\sum_{n=1}^{\infty} \frac{q^{k n}\alpha_n}{(1-q^n)^{2k}}.
\end{align*}  \label{ThmGarvan}
\end{theorem}

We now take the Bailey pair $(\alpha_n, \beta_n)$:
\begin{align*}
\alpha_{n}(a,q) &=\frac{(1-aq^{2n})(a)_n}{(1-a)(q)_n} (-1)^n a^{r n} q^{n(n-1)/2+r n^2}\\
\beta_n(a,q) &= \sum_{n\ge n_1\ge  \cdots \ge n_{r-1} \ge 0} \frac{a^{n_1+\cdots+n_{r-1}} q^{n_1^2+\cdots +n^2_{r-1}}}{(q)_{n-n_1}(q)_{n_1-n_2}\cdots (q)_{n_{r-1}}} 
\end{align*}
Let $a=1$. Then
\begin{align*}
\alpha_n(1,q)&=\begin{cases} 1 & \text{if $n=0$},\\
(-1)^n q^{n(n-1)/2+r n^2} (1+q^n) & \text{if $n\ge 1$},
\end{cases}\\
\beta_n(1,q)&=\begin{cases} 1 & \text{if $n=0$}, \\
{ \sum_{n\ge n_1\geq\cdots\geq n_{r-1}\geq 0} \frac{q^{n_1^2+\cdots +n^2_{r-1}}}{(q)_{n-n_1}(q)_{n_1-n_2}\cdots (q)_{n_{r-1}}} } & \text{if $n\ge 1$}.
\end{cases}\\
\end{align*}
Substituting $(\alpha_n(1,q), \beta_n(1,q))$  in Theorem~\ref{ThmGarvan}, we obtain
\begin{align}\label{appbp}
&\sum_{n_k\ge \cdots \ge n_1\ge 1} \frac{(q)_{n_1}^{2} q^{n_1+\cdots +n_{k}}}{(1-q^{n_1})^2 \cdots (1-q^{n_k})^2}
\sum_{m_1\geq \cdots \geq m_{r-1} \geq 0} \frac{q^{m_{1}^2+\cdots+m_{r-1}^{2}}}{(q)_{n_1-m_{1}}(q)_{m_1-m_{2}}\cdots (q)_{m_{r-1}}} \nonumber\\
&=\sum_{n_k\ge \cdots \ge n_1\ge 1} \frac{q^{n_1+\cdots +n_{k}}}{(1-q^{n_1})^2 \cdots (1-q^{n_k})^2} +\sum_{n=1}^{\infty} \frac{(-1)^{n} q^{n(n-1)/2+ r n^2+ kn} (1+q^n)}{(1-q^n)^{2k}}.
\end{align}

In the following theorem, we obtain the generating function of ${_j}\textup{spt}_k(n)$.
\begin{theorem}\label{jsptk}
We have
\begin{align}\label{gtjsptk}
& \sum_{n=1}^{\infty} {_j}\textup{spt}_k(n) q^n  \notag \\
&=\sum_{n_{k}\ge \cdots \ge n_1 \ge m_{1}\ge \cdots\ge m_{j-1} \ge 1} \frac{  q^{n_{k}+\cdots+ n_{1} +{m_{1}^2+\cdots+m_{j-1}^{2}} } (q)_{n_1}    }  {(1-q^{n_k})^2 \cdots (1-q^{n_1})^2  (q^{n_1+1})_{\infty} (q)_{n_1-m_{1}}(q)_{m_1-m_{2}}\cdots(q)_{m_{j-1}}}.
\end{align}
\end{theorem}

\begin{proof}
\iffalse
By \eqref{genjmu2k},  we see that
\begin{align*}
&\sum_{n=1}^{\infty} ({_j}\mu_{2k}(n)-{_{j+1}}\mu_{2k}(n)) q^n\\
 &= \frac{1}{(q)_{\infty}} \left( \sum_{n=-\infty \atop n\neq 0}^{\infty} \frac{(-1)^{n-1} q^{n((2j-1)n+1)/2+kn}}{(1-q^n)^{2k}} -\sum_{n=-\infty \atop n\neq 0}^{\infty}  \frac{(-1)^{n-1} q^{n((2j+1)n+1)/2 +kn}}{(1-q^n)^{2k}} \right)\\
&= \frac{1}{(q)_{\infty}} \left( \sum_{n=-\infty \atop n\neq 0}^{\infty} \frac{(-1)^{n-1} q^{n((2j-1)n+1)/2+kn} (1-q^{n^2})}{(1-q^n)^{2k}} \right). 
\end{align*}
\fi

Substituting $r=j$ and $j-1$ in \eqref{appbp}, we obtain
\begin{align}\label{appbp1}
&\sum_{n_k\ge \cdots \ge n_1\ge 1} \frac{(q)_{n_1}^{2} q^{n_1+\cdots +n_{k}}}{(1-q^{n_1})^2 \cdots (1-q^{n_k})^2}
\sum_{m_1\geq\cdots\geq m_{j-1}\geq 0} \frac{q^{m_{1}^2+\cdots+m_{j-1}^{2}}}{(q)_{n_1-m_{1}}(q)_{m_1-m_{2}}\cdots(q)_{m_{j-1}}}\nonumber\\
&=\sum_{n_k\ge \cdots \ge n_1\ge 1} \frac{q^{n_1+\cdots +n_{k}}}{(1-q^{n_1})^2 \cdots (1-q^{n_k})^2} +\sum_{n=1}^{\infty} \frac{(-1)^{n} q^{n(n-1)/2+jn^2+ kn} (1+q^n)}{(1-q^n)^{2k}}
\end{align}
and
\begin{align}\label{appbp2}
&\sum_{n_k\ge \cdots \ge n_1\ge 1} \frac{(q)_{n_1}^{2} q^{n_1+\cdots +n_{k}}}{(1-q^{n_1})^2 \cdots (1-q^{n_k})^2}
\sum_{m_1\geq\cdots\geq m_{j-2}\geq 0} \frac{q^{m_{1}^2+\cdots+m_{j-2}^{2}}}{(q)_{n_1-m_{1}}(q)_{m_1-m_{2}}\cdots(q)_{m_{j-2}}}\nonumber\\
&=\sum_{n_k\ge \cdots \ge n_1\ge 1} \frac{q^{n_1+\cdots +n_{k}}}{(1-q^{n_1})^2 \cdots (1-q^{n_k})^2} +\sum_{n=1}^{\infty} \frac{(-1)^{n} q^{n(n-1)/2+(j-1)n^2+ kn} (1+q^n)}{(1-q^n)^{2k}}.
\end{align}
Subtracting (\ref{appbp2}) from (\ref{appbp1}), we have
\begin{align}\label{appbp3}
&\sum_{n_k\ge \cdots \ge n_1\ge 1} \frac{(q)_{n_1}^{2} q^{n_1+\cdots +n_{k}}}{(1-q^{n_1})^2 \cdots (1-q^{n_k})^2}
\sum_{m_1\geq\cdots\geq m_{j-1}\geq 0} \frac{q^{m_{1}^2+\cdots+m_{j-1}^{2}}}{(q)_{n_1-m_{1}}(q)_{m_1-m_{2}}\cdots(q)_{m_{j-1}}}\nonumber\\
&\quad-\sum_{n_k\ge \cdots \ge n_1\ge 1} \frac{(q)_{n_1}^{2} q^{n_1+\cdots +n_{k}}}{(1-q^{n_1})^2 \cdots (1-q^{n_k})^2}
\sum_{m_1\geq\cdots\geq m_{j-2}\geq 0} \frac{q^{m_{1}^2+\cdots+m_{j-2}^{2}}}{(q)_{n_1-m_{1}}(q)_{m_1-m_{2}}\cdots(q)_{m_{j-2}}}\nonumber\\
&=\sum_{n=1}^{\infty} \frac{(-1)^{n} q^{n(n-1)/2+jn^2+ kn} (1+q^n)}{(1-q^n)^{2k}}-\sum_{n=1}^{\infty} \frac{(-1)^{n} q^{n(n-1)/2+(j-1)n^2+ kn} (1+q^n)}{(1-q^n)^{2k}}.
\end{align}

First, the right-hand side of (\ref{appbp3}) can be written as
\begin{align}\label{appbprhs}
&\sum_{n=-\infty \atop n\neq 0}^{\infty}\frac{(-1)^{n} q^{n(n-1)/2+jn^2+ kn}}{(1-q^n)^{2k}}-\sum_{n=-\infty \atop n\neq 0}^{\infty}\frac{(-1)^{n} q^{n(n-1)/2+(j-1)n^2+ kn}}{(1-q^n)^{2k}}\nonumber\\
&=\sum_{n=-\infty \atop n\neq 0}^{\infty}\frac{(-1)^{n-1} q^{n((2j-1)n+1)/2 +kn}}{(1-q^n)^{2k}}-\sum_{n=-\infty \atop n\neq 0}^{\infty}\frac{(-1)^{n-1} q^{n((2j+1)n+1)/2 +kn}}{(1-q^n)^{2k}}\nonumber\\
&=(q)_{\infty}\sum_{n=1}^{\infty}({_j}\mu_{2k}(n)-{_{j+1}}\mu_{2k}(n))q^n\nonumber\\
&=(q)_{\infty}\sum_{n=1}^{\infty}{_j}\textup{spt}_{k}(n)q^n,
\end{align}
where we invoked (\ref{genjmu2k}) in the penultimate step.
Also, the left-hand side of (\ref{appbp3}) can be written as
\begin{align}\label{appbplhs}
&\sum_{n_k\ge \cdots \ge n_1\ge 1} \frac{(q)^2_{n_1} q^{n_1+\cdots +n_{k}}}{(1-q^{n_1})^2 \cdots (1-q^{n_k})^2}
   \sum_{m_1\geq\cdots\geq m_{j-1}\geq 1} \frac{q^{m_{1}^2+\cdots+m_{j-1}^{2}}}{(q)_{n_1-m_{1}}(q)_{m_1-m_{2}}\cdots(q)_{m_{j-1}}}   \nonumber\\
&=\sum_{n_{k}\ge \cdots \ge n_1 \ge m_{1}\ge \cdots\ge m_{j-1} \ge 1} \frac{  (q)^2_{n_1} q^{n_{k}+\cdots+ n_{1} +{m_{1}^2+\cdots+m_{j-1}^{2}} }   }  {(1-q^{n_k})^2 \cdots (1-q^{n_1})^2 (q)_{n_1-m_{1}}(q)_{m_1-m_{2}}\cdots(q)_{m_{j-1}}},
\end{align}
since $(q)_{n_1-m_1}=0$ unless $n_1\geq m_1$. Thus, from (\ref{appbprhs}) and (\ref{appbplhs}), we have
\begin{align}\label{gtjsptk1}
&(q)_{\infty}\sum_{n=1}^{\infty} {_j}\textup{spt}_k(n) q^n\nonumber\\
&=\sum_{n_{k}\ge \cdots \ge n_1 \ge m_{1}\ge \cdots\ge m_{j-1} \ge 1} \frac{  (q)^2_{n_1} q^{n_{k}+\cdots+ n_{1} +{m_{1}^2+\cdots+m_{j-1}^{2}} }   }  {(1-q^{n_k})^2 \cdots (1-q^{n_1})^2 (q)_{n_1-m_{1}}(q)_{m_1-m_{2}}\cdots(q)_{m_{j-1}}}.
\end{align}
Finally, dividing both sides of (\ref{gtjsptk1}) by $(q)_{\infty}$, we arrive at (\ref{gtjsptk}).
\end{proof}

\textbf{Remarks.} 1. Equation (\ref{gtjsptk}) can also be written as
\begin{align}
& \sum_{n=1}^{\infty} {_j}\textup{spt}_k(n) q^n \notag \\
&= \sum_{n_{k+j-1}\ge \cdots \ge n_{j} \ge \cdots \ge n_1\ge 1} 
\frac{q^{n_{k+j-1}+\cdots +n_{j}}}{(1-q^{n_{k+j-1}})^2\cdots (1-q^{n_{j}})^2 (q^{n_{j}+1})_{\infty} }\left[\begin{matrix} n_{j}\\ n_{j-1} \end{matrix}\right] \cdots \left[\begin{matrix} n_2\\ n_1 \end{matrix}\right] q^{n_{j-1}^2+\cdots+ n_1^2}.  \label{remark}
\end{align}
Also, note that ${_1}\textup{spt}_k (n)=\textup{spt}_k(n)$.\\

2. From (\ref{sptmot}) and (\ref{def3.4}), we have
\begin{equation}
{}_j\textup{spt}_{1}(n)=\textup{Spt}_{j}(n)-\textup{Spt}_{j-1}(n),
\end{equation}
which implies that
\begin{equation}
\textup{Spt}_{j}(n)=\sum_{\ell=1}^{j}{}_{\ell}\textup{spt}_{1}(n).
\end{equation}
This in turn gives $\textup{Spt}_{j}(n)={}_1\mu_{2}(n)-{}_{j+1}\mu_{2}(n)$.

\section{A combinatorial interpretation of ${_j}\textup{spt}_k(n)$}

We recall the higher order spt-function $\textup{spt}_k(n)$ studied by Garvan. For a partition $\pi$, let 
\begin{align}
w_k(\pi)=\sum_{m_1+\cdots +m_r=k \atop 1\le r\le k} \binom{f_{t_1}+m_1-1}{2m_1-1} \sum_{t_1< t_2<t_3\cdots <t_r} \binom{f_{t_2}+m_2}{2m_2} \cdots \binom{f_{t_r}+m_r}{2m_r}, \label{wk}
\end{align}
where the outer sum is over all compositions $m_1+\cdots +m_r$ of $k$, $t_1, t_2,\ldots, t_r$ are distinct parts of $\pi$ with the smallest part $t_1$, and  $f_{t}$ denotes the number of occurrences of $t$ in the partition $\pi$. He then defined  $\text{spt}_k(n)$ as 
\begin{align*}
\textup{spt}_k(n)=\sum_{\pi} w_{k}(\pi),
\end{align*}
where the sum is over all partitions $\pi$ of $n$, and showed that the generating function of $\textup{spt}_{k}(n)$ is given by (\ref{garvansptk}).
%\begin{align}
%\sum_{n=1}^{\infty}spt_k(n)q^n=\sum_{n_{k}\ge \cdots \ge n_{1} \ge 1} 
%\frac{q^{n_{k}+\cdots +n_{1}}}{(1-q^{n_{k}})^2\cdots (1-q^{n_{1}})^2 (q^{n_{1}+1})_{\infty} }. \label{garvansptk}
%\end{align}

We now generalize this to ${_j}\textup{spt}_k(n)$.  Let $j,k\ge 1$. For a partition $\pi$, we define a weight
\begin{align*}
{_j}w_k(\pi)=\sum_{t_1} \sum_{m_1+\cdots +m_r=k\atop 1\le r\le k} \binom{f'_{t_1}+m_1-1}{2m_1-1} \sum_{t_1< t_2<\cdots <t_r} \binom{ f_{t_2} + m_2}{2 m_2} \cdots \binom{f_{t_r} + m_r}{2 m_r},
\end{align*}
where the outer sum is over all parts $t_1$ right above each of the parts contained in the $(j-1)$st lower-Durfee square, and the middle sum  is over all compositions $m_1+\cdots +m_r$ of $k$, $t_2,\ldots, t_r$ are distinct parts of $\pi$ greater than $t_1$ and $f'_t$ denotes the mark of $t$. Then, we obtain the following theorem.

\begin{theorem}
With ${_j}\textup{spt}_k$ defined in \eqref{def3.4},  we have
\begin{align*}
{_j}\textup{spt}_k(n)=\sum_{\pi} {_j}w_k(\pi),
\end{align*} 
where the sum is over all partitions of $n$. 
\end{theorem}

\proof 
We take the right hand side of \eqref{remark}:
\begin{align}
\sum_{n_{k+j-1}\ge \cdots \ge n_{j} \ge \cdots \ge n_1\ge 1} 
\frac{q^{n_{k+j-1}+\cdots +n_{j}}}{(1-q^{n_{k+j-1}})^2\cdots (1-q^{n_{j}})^2 (q^{n_{j}+1})_{\infty} }\left[\begin{matrix} n_{j}\\ n_{j-1} \end{matrix}\right] \cdots \left[\begin{matrix} n_2\\ n_1 \end{matrix}\right] q^{n_{j-1}^2+\cdots+ n_1^2}. \label{remarkright}
\end{align}
Then, we see that
\begin{align*}
\sum_{n_{j} \ge \cdots \ge n_1\ge 1} 
\left[\begin{matrix} n_{j}\\ n_{j-1} \end{matrix}\right] \cdots \left[\begin{matrix} n_2\\ n_1 \end{matrix}\right] q^{n_{j-1}^2+\cdots+ n_1^2}
\end{align*}
generates partitions $\mu$ into parts less than or equal to $n_j$ with exactly $j-1$ successive Durfee squares.  Also, it follows from \eqref{garvansptk} that
\begin{align*}
\sum_{n_{k+j-1}\ge \cdots \ge n_{j}\ge 1} 
\frac{q^{n_{k+j-1}+\cdots +n_{j}}}{(1-q^{n_{k+j-1}})^2\cdots (1-q^{n_{j}})^2 (q^{n_{j}+1})_{\infty} }.
\end{align*}
generates weighted partitions $\nu$ with the smallest part equal to $n_j$ and weight $w_k(\nu)$ defined in \eqref{wk}.

%That is,
%\begin{align*}
%w_k(\nu) = \sum_{m_1+\cdots +m_r=k \atop 1\le r\le k} \binom{f_{n_j}+m_1-1}{2m_1-1} \sum_{n_j< n_{j+1}<\cdots <n_{j+k-1}} \binom{f_{n_{j+1}}+m_2}{2m_2} \cdots \binom{f_{n_{j+k-1}}+m_r}{2m_r}.
%\end{align*}
Clearly, the union of $\mu$ and $\nu$ is generated by \eqref{remarkright} with weight $w_k(\nu)$.  With the same argument in the proof of Theorem~\ref{sptwkpi}, a partition $\pi$ generated by \eqref{remarkright} can be split into such $\mu$ and $\nu$ by separating the parts above any part in its $(j-1)$st successive lower-Durfee square.  That is, the part right above each of the parts in the $(j-1)$st successive lower-Durfee square can be the smallest part of $\nu$. In addition, the number of occurrences of the smallest part in $\nu$ is equal to its mark in $\pi$, namely $f_{n_j}$ in $\nu$ equals $f'_{n_j}$ in $\pi$.  Thus, the partition $\pi$ is generated with weight ${_j}w_k(\pi)$ as desired. 
\endproof

\section{Inequality between the moments of $j$-rank and $(j+1)$-rank}
Since the idea is completely analogous to the one used for proving $M_{2k}(n)>N_{2k}(n)$, we just give the main results below. With ${}_j\mu_{k}(n)$ defined in (\ref{ksymj}), we have the following:
\begin{theorem}
For $k\geq 1$,
\begin{align}\label{relos}
{}_j\mu_{2k}(n)&=\frac{1}{(2k)!}\sum_{m=-n}^{n}g_{k}(m)N_{j}(m,n),\nonumber\\
{}_jN_{2k}(n)&=\sum_{t=1}^{k}(2t)!S^{*}(k,t){}_j\mu_{2t}(n).
\end{align}
\end{theorem}
\begin{proof}[Theorem \textup{\ref{genineq}}][]
Suppose $k\geq 1$. From (\ref{remark}) and (\ref{def3.4}), we have
\begin{equation}
\sum_{n=1}^{\infty}\left({}_j\mu_{2t}(n)-{}_{j+1}\mu_{2t}(n)\right)q^n=\frac{q^{t+j-1}}{(1-q)^{2(t+j-1)}}\cdot\frac{q^{j-1}}{(q^2)_{\infty}}+\cdots,
\end{equation}
and hence
\begin{equation}
{}_j\mu_{2t}(n)>{}_{j+1}\mu_{2t}(n),
\end{equation}
for all $n\geq t\geq 1$ and $j\geq 1$. Using (\ref{relos}) and the fact that $S^{*}(k,t)$ are positive integers, we have
\begin{equation}
{}_jN_{2k}(n)-{}_{j+1}N_{2k}(n)=\sum_{t=1}^{k}(2t)!S^{*}(k,t)\left({}_j\mu_{2t}(n)-{}_{j+1}\mu_{2t}(n)\right)\geq 2\left({}_j\mu_{2t}(n)-{}_{j+1}\mu_{2t}(n)\right)>0,
\end{equation}
for all $n\geq 1$. 
\end{proof}
A simple consequence of Theorem \ref{genineq} is that $M_{2k}(n)={}_1\mu_{2k}(n)>{}_j\mu_{2k}(n)$ for all $j>1$.\\

\textbf{Acknowledgements.} The authors sincerely thank Bruce C.~Berndt for several suggestions which improved the quality of this paper. This work was done while the second author was visiting University of Queensland. She thanks Ole Warnaar for his warm hospitality.

\end{document}